\documentclass[12pt,reqno]{amsart}

\usepackage{natbib}
\usepackage{amsmath,amssymb,latexsym} % ams auxiliaries.
\usepackage{fullpage}
\usepackage{etoolbox}
\usepackage{enumitem}
\usepackage{color}
\usepackage[colorlinks,linkcolor=blue,citecolor=magenta]{hyperref}
\usepackage[colorinlistoftodos,bordercolor=orange,backgroundcolor=orange!20,linecolor=orange,textsize=scriptsize]{todonotes}
\usepackage{soul}
\usepackage{microtype}
\usepackage{yhmath}
\usepackage{float}
\usepackage{algorithm}
\usepackage[noend]{algpseudocode}
\usepackage{bbold}
\usepackage{booktabs}
\usepackage{comment}
\usepackage{mathtools}
\usepackage{calc}
\newlength{\LPlhbox}

\newtheorem{theorem}{Theorem}
\newtheorem{corollary}{Corollary}
\newtheorem{proposition}{Proposition}

\newtheorem{lemma}{Lemma}

\theoremstyle{definition}
\newtheorem{remark}{Remark}

\def\leq{\leqslant}
\def\geq{\geqslant}

\def\R{\mathbb{R}}
\def\Z{\mathbb{Z}}
\def\P{\mathcal{P}}
\def\S{\mathcal{S}}
\def\x{\boldsymbol{x}}
\def\y{\boldsymbol{y}}

\def\u{\boldsymbol{u}}
\def\a{\boldsymbol{a}}
\def\cc{\boldsymbol{c}}
\def\eell{\boldsymbol{\ell}}
\def\ds{\displaystyle}
\def\zero{\boldsymbol{0}}
\def\uun{\hat{\boldsymbol{1}}}

\def\leqlex{\leq_{\text{lex}}}
\def\geqlex{\geq_{\text{lex}}}
\def\leqslex{<_{\text{lex}}}
\def\geqslex{>_{\text{lex}}}

\newcommand{\lexgeq}{\geqslant_{\mathrm{lex}}}

\newcommand{\lexmin}{\operatorname{lexmin}}

\DeclareMathOperator{\lexmax}{lexmax}

\newcommand{\LPblocktag}[2]{\settowidth{\LPlhbox}{(#1)}%
	\parbox{\LPlhbox}{\begin{equation}\tag{#1}#2\end{equation}}%
	\hspace*{\fill}}
	
\usepackage{commands}

\title{Linear lexicographic optimization and preferential bidding system}

\author{Nour ElHouda Tellache}
\address{N.E.H. Tellache, College of Science and Engineering, Hamad Bin Khalifa University, Doha, Qatar}
\email{nour.tellache@gmail.com}

\author{Fr\'ed\'eric Meunier}
\address{F. Meunier, CERMICS, \'Ecole des Ponts, 77455 Marne-la-Vall\'ee CEDEX, France}
\email{frederic.meunier@enpc.fr}

\author{Axel Parmentier}
\address{A. Parmentier, CERMICS, \'Ecole des Ponts, 77455 Marne-la-Vall\'ee CEDEX, France}
\email{axel.parmentier@enpc.fr}


\begin{document}

\maketitle

\begin{abstract}
Some airlines use the preferential bidding system to construct the schedules of their pilots. In this system, the pilots bid on the different activities and the schedules that lexicographically maximize the scores of the pilots according to their seniority are selected. A sequential approach to solve this maximization problem is natural: the problem is first solved with the bids of the most senior pilot; then it is solved with those of the second most senior without decreasing the score of the most senior, and so on. The literature admits that the structure of the problem somehow imposes such an approach.

The problem can be modeled as an integer linear lexicographic program. We propose a new exact method, which relies on column generation for solving its continuous relaxation. To design this column generation, we prove that bounded linear lexicographic programs admit ``primal-dual'' feasible bases and we show how to compute such bases efficiently. 

Another contribution on which our exact method relies consists in the extension of standard tools for resource-constrained longest path problems to their lexicographic versions. This is useful in our context since the generation of new columns is modeled as a lexicographic resource-constrained longest path problem.

Numerical experiments show that this new method is already able to solve industrial instances provided by Air France, with up to $150$ pilots. By adding a last ingredient in the resolution of the longest path problems, which exploits the specificity of the preferential bidding system, the method achieves for these instances computational times that are compatible with operational constraints.
\end{abstract}

\section{Introduction}

\subsection{Context}
The crew scheduling problem is one of the most important problems in the airline planning process because the total crew cost (salaries, benefits, and expenses) is considered, next to the fuel cost, the largest single cost of an airline~\citep{Gopalakrishnan2005}. For large fleets, this problem is decomposed into crew pairing and crew assignment, which are solved sequentially~\citep{Kasirzadeh2017}. The first consists of generating a set of least-cost crew pairings (sequences of flights starting and ending at the same crew base) that cover all the flights. The second aims at finding monthly schedules (sequences of pairings) for crew members that cover the pairings previously built. The integrated crew pairing and crew assignment problem has also been studied in the literature; see, e.g., the work by~\cite{ZEIGHAMI2020211}.

Different scheduling processes exist to construct the schedules of the crew members. With the bidline process, anonymous schedules are first computed (various objectives can be considered). Crew members bid on these schedules and are later assigned to them according to their bids and seniority~\citep{BOUBAKER201050}. The rostering process consists of computing personalized schedules, taking into account preassigned activities such as vacations and training periods while optimizing an objective function such as the cost, the covered pairings, the global satisfaction, etc.~\citep{Gamache1999, kohl2004}. In the preferential bidding system (PBS), that is the subject of the present paper, the crew members express their preferences on the various activities (pairings, pre-assignments, rest periods, etc.) in the form of bids. Then, the system  builds the schedules that maximize the scores of the crew members in a lexicographic order: the most senior is served first, then the second one, etc. The score of a schedule for a crew member is calculated from his bids.

As discussed hereafter, designing a method that outputs a solution satisfying exactly the requirements of the preferential bidding system is a difficult problem, and most of the literature on this problem focuses on its application to pilots. In this paper, we also consider the case of pilots.

\subsection{Challenges of exact methods} The number of feasible schedules grows exponentially with the number of pairings, and this clearly makes the PBS problem difficult. Yet, classical techniques from mathematical programming are available to deal efficiently with such issues (e.g., column generation), but what makes it non-standard is its lexicographic objective function. Two exact approaches for dealing with such an objective function %\comAP{Removed: come quite naturally to mind and} 
have been considered in the literature: a ``sequential'' approach and a ``weighting'' approach. 

The principle of a sequential approach goes as follows. First, the largest possible score that is achievable by the most senior pilot is determined: this is done by solving the PBS problem with the objective function restricted to his preferences. Then, the largest possible score that is achievable by the second most senior pilot, without decreasing that of the most senior, is determined: this is done by solving the problem with the objective function restricted to the preferences of the second most senior, and by restricting the set of feasible solutions to those yielding the optimal score for the most senior. And so on.

A weighting approach consists in considering a single objective function where the score of each pilot gets a weight. This way, we replace the original integer linear lexicographic program by a classical integer linear program (e.g., with a single objective function). The challenge here is to find the weights so that the optimal solutions of the two programs are the same. Such weights exist: indexing the pilots from $1$ to $m$ from the most senior to the least senior, %assuming that the scores are integer numbers, 
choose for pilot $i$ a weight equal to $M^{m-i}$, where $M$ is any upper bound on the maximal achievable score over all pilots. To be valid in full generality, these weights must actually increase exponentially with the number $m$ of pilots. \cite{GamachePBS1998} and \cite{achour2007} note that numerical precision issues prevent any use of this approach for realistic number of pilots and they conclude to the necessity of the sequential approach.

The two approaches reduce then to the classical resolution of integer programs with a single objective function, which are efficiently solved via column generation for the continuous relaxation and via branch-and-bound to get optimal integer solutions.

\subsection{Contributions} 

We propose a third alternative to the sequential and weighting approaches.% and this alternative comes under the scope of a classical integer programming paradigm. 

The key element of our approach is the design of a new method for solving linear lexicographic programs, which relies on column generation. Two preliminary theoretical results have been required to design such a column generation. \cite{Isermann1982} has shown how a natural extension of the definition of reduced costs leads to a simplex-like method for solving linear lexicographic programs. Similarly to the classical setting, the reduced cost of an entering variable indicates its potential improvement on the whole lexicographic objective. Our first theoretical result (Corollary~\ref{cor:dual-basic}) is that, with this extended definition of reduced costs, bounded linear lexicographic programs admit ``primal-dual'' feasible bases. The existence of such bases opens the door to column generation methods. \cite{akgul1984} sketched in a short note a promising way to compute an optimal basis of a linear lexicographic program, relying on specific properties of standard linear programming. Our second theoretical result (Theorem~\ref{thm:lex-reg}) completes his idea by showing that the obtained basis is actually primal-dual feasible. Thanks to this, we can use current off-the-shelf solvers to find primal-dual feasible bases for linear lexicographic programs and avoid re-implementing any simplex-like method. We emphasize that, although most of these solvers can solve such lexicographic programs directly, we do not use this option because none of them give access to reduced costs or primal-dual feasible bases (in the lexicographic sense). It seems that they proceed by adding constraints, and it is not clear how to retrieve these elements.
%This is crucial for our approach because neither reduced costs nor primal-dual feasible bases (in the lexicographic sense) can be retrieved by the current off-the-shelf solvers, although most of these latter are able to solve linear lexicographic programs.

In the case of the PBS problem, the task of generating new columns (the ``pricing'' problem) takes the form of a resource-constrained longest path problem with lexicographic costs. We show that a classical algorithm based on bounds and dominance can solve this problem. We eventually achieve excellent performances by exploiting the specific form of the pricing problem (the ``reduction'' trick).

%We have also shown how standard tools at the root of efficient algorithms for resource-constrained longest path problems can be extended to their lexicographic versions. This is useful in our context since the generation of new columns is modeled as a lexicographic resource-constrained longest path problem.

%\cite{Isermann1982} has proved that several results of linear programming extend to linear lexicographic programming just by replacing the ordinary order on $\R$ by the lexicographic order on $\R^m$. We show that this extension allows continuous relaxation to provide bounds for integer lexicographic programs in the same way it provides bounds for ordinary integer programs, and to design a column generation method based on dual information, similarly to ordinary column generation.

For the overall exact method we propose for the PBS problem, we follow a standard approach relying on the efficient resolution of the continuous relaxation via column generation (see, e.g., the work by \cite{Valentina2016}). It goes as follows. First, a continuous relaxation is solved via column generation, providing an upper bound to the optimal value. This continuous relaxation takes the form of a linear lexicographic program (the ``master'' problem) and, as usual, the game consists in identifying columns with positive reduced costs (the ``pricing'' problem). Second, a lower bound is computed with an IP solver that can deal with lexicographic objectives (like Gurobi). Third, columns ``in the gap,'' that is new columns that might improve the value of the solution, are generated (and identified as such via their reduced costs), and the solver is called again. We emphasize that all comparisons involved in expressions like ``upper bound,'' ``lower bound,'' ``positive,'' and ``in the gap'' are understood according to the lexicographic order.

Experiments have been conducted on instances provided by the French airline Air France and have proved the efficiency of the method: most of the realistic instances we worked with have been solved exactly in around $1$ hour, while having up to $100$ pilots. Such instances are already considered as quite large by Air France---and anyway are the largest considered in the literature in terms of number of pilots---and the computational times are compatible with their operational constraints. Our method is thus applicable in practice. We also believe that what makes this method appealing is not only its efficiency but also its conceptual simplicity.

So, we can summarize our main contribution not only as an efficient method for solving the PBS problem but also as the proof that standard schemes of mathematical programming (such as column generation, bounds, dominance, etc.) can be extended directly to optimization problems with lexicographic objectives. %In particular, we are not aware of any work in the literature dealing with column generation for linear programs with such objectives.

\section{Literature review}

In accordance with our approach, the literature review is split into two parts: solving the PBS problem; solving linear lexicographic programs. We have not been able to find relevant works on shortest or longest paths in a graph with lexicographic costs. 

\subsection{PBS problem}
The PBS problem has been less studied in the literature in comparison to the other crew assignment systems. The earliest papers dealing with this problem use heuristics that sequentially construct the schedules of the crew members following a seniority order. \cite{Moore1978} and \cite{Byrne1988} have proposed greedy heuristics for Qantas that build a feasible schedule for each crew member, from the most senior to the most junior, by selecting iteratively the pairing with the highest score among the residual pairings. Feasibility checks are performed in order to ensure that the iteration can yield a feasible schedule and that the problem remains feasible for the less senior crew members. 
%This latter check is done by verifying that for every day of the considered period, the number of the remaining crew members is greater than or equal to the number of residual pairings. The uncovered pairings at the end of the process are declared open and can be bid for outside the preferential bidding system. \todo{check if this latter sentence should be included since I dont know if if the paper of Moore et al. have used the same idea to check for feasibility}. 

\cite{GamachePBS1998} have proposed another sequential heuristic approach for Air Canada that associates to each crew member, starting from the most senior, a mixed integer linear program. This program consists in finding a maximal-score schedule for the current crew member, while taking into account the schedules assigned to the most senior crew members and ensuring that the less senior ones can cover the remaining pairings. This latter feasibility condition is taken into account in an approximate way by relaxing the decision variables corresponding to the less senior crew members. As a consequence, some programs may not be feasible which requires time-consuming backtracking to correct bad decisions. Each program is solved using a column generation algorithm embedded in a branch-and-bound framework. A second branching tree is used to maintain the links between the different programs in case of backtracking. This heuristic has been tested on $24$ instances with size up to $108$ pilots and the most difficult instance requires about $8$ hours to be solved. In another paper, \cite{GAMACHE20072384} added to this heuristic approach a feasibility test to avoid backtracks. This test determines at each iteration if it is possible to assign a feasible schedule to each crew member while covering all pairings. The feasibility test problem is modeled as a graph coloring problem and solved using a tabu search algorithm. The proposed method has been validated on an instance from a North American carrier with $330$ crew members (not only pilots).

Since several schedules may yield the same score for a given crew member, fixing a schedule for that crew may reduce the best scores of less senior crew members. For this reason, the approach of \cite{GamachePBS1998} is not exact. A few years later, \cite{achour2007} proposed an exact algorithm that adds to the heuristic of \cite{GamachePBS1998} a mechanism delaying the selection of a best-score schedule for a senior crew member until there is only one schedule left for that crew that can yield his best score. To do so, after solving to optimality the program of a given crew member, the set of all residual best-score schedules for that crew is computed. The sets corresponding to the best-score schedules of the most senior crew members and the sets corresponding to the feasible schedules of the junior ones are updated accordingly. If we are left with one best-score schedule for a senior crew member, a feasibility check is performed before assigning that schedule. This is done by solving the mixed integer linear program of the current crew member in which the integrality is imposed on the decision variables of all senior crew members. This program is also solved  using a branch-and-price algorithm. If the problem is not feasible, the program of the current crew member is solved again by adding a cut restricting the score of that crew. Recall that, similarly to the heuristic of \cite{GamachePBS1998}, the overall approach may result in non-feasible programs for some crew members. A backtracking is applied, in this case, using the second branching tree. This exact algorithm has been tested and compared with the heuristic of~\cite{GamachePBS1998} on $8$ instances corresponding to the PBS problem for the pilots at a North American carrier. The size of the instances varies from $17$ to $91$ pilots, and the most difficult instance requires more than $5$ days with the exact algorithm while the heuristic spends more than $7$ days on that instance.

%Gurobi solves a linear lexicographic program of $m$ objective functions by performing $m$ separate optimization steps. In each step, in decreasing priority order, it optimizes for the current objective multiplied by its $\mathsf{ObjNWeight}$ attribute, while imposing constraints that ensure that the quality of higher-priority objectives isn't degraded by more than the specified tolerances. By giving the same weight to all the objectives and by imposing null degradation tolerance, we obtain an optimal solution to the lexicographic linear program. The solver also handle the case of objectives with the same priority by taking a linear combination of those objectives multiplied by their weights, this results in fewer than $m$ total steps for the $m$ objectives. 

Column generation in a lexicographic context has been recently applied by \cite{katrin2022} on a different problem. The authors studied a bin packing problem with five objective functions to be optimized in a lexicographic order. They followed a sequential approach and developed a branch-and-price algorithm for each single-objective stage. The pricing problems were formulated as shortest path problems with resource constraints and solved by dynamic-programming labeling algorithms. 

\subsection{Linear lexicographic optimization}
Most of the literature on linear lexicographic optimization focuses on sequential or weighted approaches to deal with lexicographic objectives. To our knowledge, the earliest paper that considers such problems and suggests a different approach is the work of \cite{Isermann1982}. He laid the theoretical foundation of linear lexicographic optimization, which is used to extend the ordinary simplex method to a lexicographic context where all the objectives are considered simultaneously. %He provides the existence and duality theory which are the theoretical basis of the solution method. 
No numerical experiments are presented in the paper and the algorithm is only illustrated on an example. \cite{akgul1984} has sketched a way to solve linear lexicographic programs that considers each objective sequentially and restricts the problem at each iteration to the variables that finish the previous iteration with zero reduced costs. This prevents from adding a new constraint at each iteration, which is a natural trick to ensure the lexicographic optimization. This is actually the method we have followed for solving the master problem; see Section~\ref{sec:master}.

\cite{POURKARIMI20071348} described another approach in which they solve a single linear program. This latter corresponds to the dual of the weighted problem associated with the linear lexicographic program. The algorithm is tested on an example.

Recently, \cite{COCOCCIONI2018298} proposed a solution method that also considers all the objectives at the same time. This approach adds to the usual numeral system a new numeral, which they call ``grossone'' and which is defined as the number of elements of the set of natural numbers. (Curiously, they do not refer to the standard notion of $\aleph_0$.) This allows the authors to transform the lexicographic objective into a single objective, with ``infinitesimal'' weights whose order decreases with the decrease of the importance of the objectives. The obtained program is solved using a simplex-like method (called GrossSimplex) working with grossone-based numbers that can include infinitesimal parts. The proposed method has been tested on $4$ examples with at most $10$ objective functions. The same authors addressed~\citep{COCOCCIONI2020105177} the case in which some variables are integers. They solved the problem using a branch-and-bound algorithm in which the relaxation at each node of the search tree is solved using the GrossSimplex. The algorithm has been tested on 5 examples; 4 of them have at most 7 objective functions and the last example has 200 objectives. This later requires about 32 hours to be solved on an Intel i7 920 with cores at 3.6 GHz.

%\subsection{Lexicographic shortest or longest paths} 

\section{Problem statement and overview of the method}

\subsection{Problem formulation and modeling}\label{subsec:pb}
The problem we consider is the PBS problem for pilots. We are given a set $\P$ of pairings and a set $\S\subseteq 2^\P$ of feasible schedules. There are $m$ pilots, identified with the integers from $1$ to $m$. We denote their set by $I\coloneqq [m]$. The numbers reflect the seniority: the most senior is identified with $1$ and the least senior is identified with $m$. Each pilot $i$ provides a score $g_{ip}$ to every pairing $p$. In practice, this score is expressed indirectly by the pilot who only formulates preferences, e.g., morning flight, specific destination, etc.

The goal is to find an injective assignment $\sigma\colon I\rightarrow \S$ of the pilots to the feasible schedules so that the selected schedules $\sigma(I)$
\begin{itemize}
    \item form a partition of the pairing set $\P$, and
    \item maximize the scores of the pilots lexicographically. 
\end{itemize}
The lexicographic maximization of the scores can be formalized as follows: the vector $(c_{1\sigma(1)}, c_{2\sigma(2)},\ldots,c_{m\sigma(m)})$ must be lexicographically maximal, where $c_{is}=\sum_{p\in s}g_{ip}$ is the score of schedule $s$ according to pilot $i$.

This problem can be modeled as the following integer linear lexicographic program:

\hspace*{-15cm}
\LPblocktag{P}{\label{main-problem}}%
\begin{minipage}{\linewidth-5cm}
	\begin{flalign}\notag
\lexmax\quad & C\x &&&\notag\\
\mbox{s.t.}   \quad  &  \ds{\sum_{s\in \S}x_{is}=1} && \forall i\in I &\label{eq:assignement} \\
& \ds{\sum_{i\in I}\sum_{s \in \S \colon s \ni p} x_{is}=1} && \forall p\in \P &\label{eq:partition} \\
& x_{is}\in\{0,1\} && \forall i\in I, s\in \S.& \notag%\label{blabla}  
	\end{flalign}~
\end{minipage}

The notation ``$\lexmax\; C\x$'' means that we are looking for a feasible solution such that 
\begin{equation}\label{eq:obj}
\left( \sum_{s\in \S}c_{1s}x_{1s}, \sum_{s\in \S}c_{2s}x_{2s}, \ldots, \sum_{s\in \S}c_{ms}x_{ms}\right)
\end{equation}
is lexicographically maximal.

Even if it is not necessary at this stage, since we have a well-defined mathematical program, we remark that the notation $C\x$ does also make sense with the following definition of $C$. It is a matrix with rows identified with $I$ and with columns identified with $I \times \S$. The entry on row $i$ and column $(j,s)$ is equal to $c_{is}$ if $i=j$ and $0$ otherwise. This way, $C\x$ is exactly the (transpose of the) vector given in equation~\eqref{eq:obj}.

In the specific version of the PBS problem studied in this paper, the set $\S$ is only given implicitly. Pairings are considered over a given month. A pairing comes together with many characteristics. In the paper, we focus on the main ones: starting and ending times, %duration of the rest periods before and after, 
number of flight hours, and working hours. A subset $s$ of the pairing set $\P$ belongs to $\S$ if the following constraints are satisfied (a day in a schedule is {\em on} or {\em off} depending on whether there is some work on that day or not):
\begin{itemize}
    \item No two pairings in $s$ overlap in time.
    %\item No pairing overlaps the rest period of any other pairing.
    \item The total number of days on is at most $17$ (in case it is a $30$-day month). 
     \item The days off must include at least $7$ consecutive days. (Here, a day starts at $12$am and finishes at $11{:}59$pm.)
    \item The total number of flight hours does not exceed $85$ hours.
    \item The total number of working hours does not exceed $55$ for every sequence of $7$ consecutive days.
\end{itemize}
%(Actually, other technical constraints have been taken into account in the implementation but for sake of readability we omit the details.)

\subsection{Preliminaries on linear lexicographic programming}\label{subsec:prel}

According to \cite{Isermann1982}, linear lexicographic programming studies linear programs with a finite number of objective functions that ``are to be optimized [...] in a \emph{lexicographic order}, i.e., low priority objectives are optimized as far as they do not interfere with the optimization of higher priority objectives.'' Such linear lexicographic programs are linear programs where the order on $\bar\R=\R\cup\{-\infty,+\infty\}$ in the objective function is replaced by the lexicographic order $\leqlex$ on $\bar\R^m$, defined as follows. We have $\x \leqslex \y$ if there exists $j\in[m]$ such that $x_j < y_j$ and $x_h = y_h$ for all $h < j$. We have $\x \leqlex \y$ if $\x\leqslex\y$ or $\x=\y$.

Without loss of generality, such a linear program can be written as follows:

\hspace*{-15cm}
\LPblocktag{LLP}{\label{lin-lex-prog}}%
\begin{minipage}{\linewidth-5cm}
	\begin{flalign}\notag
\lexmax\quad & C\x &&& \notag\\
\mbox{s.t.}   \quad  &  A\x = \boldsymbol{b} && &\notag\\
& \x\geq\zero\, , \notag &&&
	\end{flalign}~
\end{minipage}

\noindent where $A$ and $C$ are respectively $k\times n$ and $m\times n$ real matrices, and $\boldsymbol{b}$ is a vector in $\R^k$. The notation ``$\lexmax\; C\x$'' means that we are looking for a feasible $\x^*$ such that $C\x^* \geqlex C\x$ for every feasible $\x$.

Assume that $A$ has full row rank. We keep the definition of basis and feasible basis of ordinary linear programming (the objective function is not involved). We also keep the standard notation $M_X$ meaning that we restrict a matrix $M$ to a subset $X$ of its columns. Given a feasible basis $B$, we define then the \emph{reduced cost of $x_j$ with respect to $B$} as $\overline \cc_j = \cc_j - C_B A_B^{-1} \a_j$, where $\a_j$ and $\cc_j$ are the $j$-th column of $A$ and $C$, respectively. In particular, $\overline \cc_j=0$ when $j\in B$. Note that the reduced cost of a variable is now an $m$-dimensional vector, and, when $m=1$, this definition coincides with the ordinary one.

 As shown by Isermann (see the paragraph before Proposition 1 in his paper), the standard criterion ensuring the optimality of a feasible basis still holds in the lexicographic setting. A basis is \emph{primal-dual feasible} if the corresponding reduced cost $\overline \cc_j$ satisfies $\overline \cc_j \leqlex \zero$ for all $j$.

\begin{lemma}\label{lem:opt-basis}
Any solution determined by a primal-dual feasible basis is optimal.
\end{lemma}

As in standard linear programming, the converse is not true: there might exist optimal solutions determined by bases that are not primal-dual feasible.

%\begin{lemma}\label{lem:opt-basis}
%Let $B$ be a feasible basis. If the reduced costs $\overline \cc_j$ with %respect to $B$ are such that $\overline \cc_j \leqlex \zero$ for all $j$, then %the feasible solution determined by $B$ is optimal.
%\end{lemma}

A sufficient condition for the existence of such a basis is given in the following result, which is a corollary of Theorem~\ref{thm:lex-reg}, stated and proved in Section~\ref{sec:master}. %Note that by Lemma~\ref{lem:opt-basis}, the feasible basis $B$ in Corollary~\ref{cor:dual-basic} determines an optimal solution. 
(The existence of a basis determining an optimal solution is proved within the proof of Lemma 2 in Isermann's paper; yet, the part regarding the dual feasibility is not present and, since it is crucial for our method to work, we provide a proof in Section~\ref{sec:master}.)

\begin{corollary}\label{cor:dual-basic}
If~\eqref{lin-lex-prog} is feasible and lexicographically bounded from above, then there exists a primal-dual feasible basis.
\end{corollary}

%\begin{lemma}\label{lem:dual-basic}
%If~\eqref{lin-lex-prog} is lexicographically bounded from above, then there %exists a feasible basis $B$ such that $\overline \cc_j=\zero$ for all $j\in B$ %and $\overline \cc_j\leqlex\zero$ for all $j\notin B$.
%\end{lemma}

In Section~\ref{sec:master}, we will show that the approach proposed by \cite{akgul1984} provides actually an efficient way to find a basis as in Corollary~\ref{cor:dual-basic}.

We end this section with a lemma providing a way to reduce the size of an integer lexicographic program. It will play an important role in our method. The lemma is actually just an extension to the lexicographic context of a classical trick for ordinary integer programming; see \cite[Section 3.3]{crowder_solving_1983} or \cite[Proposition 2.1, p. 389]{Nemhauser1988}. Actually this trick was already used by \cite{dantzig_solution_1954} for solving to optimality an instance of the traveling salesman problem with $49$ cities.

Consider the following integer linear lexicographic program.

\hspace*{-15cm}
\LPblocktag{ILLP}{\label{ILLP}}%
\begin{minipage}{\linewidth-5cm}
	\begin{flalign}\notag
\lexmax\quad & C\x &&& \notag\\
\mbox{s.t.}   \quad  &  A\x = \boldsymbol{b} && &\notag\\
& \x\in\Z_+^n\, . \notag &&&
	\end{flalign}~
\end{minipage}

\begin{lemma} \label{gap}
Let $\eell\in \bbR^m$ be a lower bound of~\eqref{ILLP} and $\u\in \R^m$ the optimal value of its linear relaxation. Consider a primal-dual feasible basis and a non-basic variable $x_{j}$. If $\overline \cc_j\leqslex\eell-\u$, then $x_{j}$ is equal to $0$ in all optimal solutions of~\eqref{ILLP}.
 \end{lemma}
\begin{proof}
Denote by $B$ the basis considered in the statement. According to Lemma~\ref{lem:opt-basis}, the solution determined by $B$ is optimal and thus $C_BA_B^{-1}\boldsymbol{b}=\u$. By definition of reduced costs, every feasible solution $\x$ of the linear relaxation satisfies therefore $C\x=\u+\sum_{j'\in N}x_{j'}\overline\cc_{j'}$, where $N$ is the set $[n]\setminus B$ of non-basic indices. The vector $\eell$ being a lower bound, we have $\eell-\u\leqlex \sum_{j'\in N}x_{j'}\overline\cc_{j'}$. Since $\overline \cc_{j'}\leqlex\zero$ for all $j'\in N$, every feasible solution $\x$ of the linear relaxation satisfies $\eell-\u\leqlex x_j \overline \cc_j$. 

Consider now an optimal solution $\x^*$ of~\eqref{ILLP} and suppose that $\overline \cc_j\leqslex\eell-\u$. According to what has just been proved, we have $\eell-\u\leqlex x^*_j \overline \cc_j$, which implies $\overline \cc_j\leqslex x^*_j \overline \cc_j$. Since $\overline \cc_j\leqlex\zero$, we have $a\overline \cc_j\leqlex\overline \cc_j$ for every positive integer $a$ (because the  lexicographic order is compatible with addition). Therefore, $x_j^*=0$, which is the desired conclusion.
\end{proof}

\subsection{Method for solving the PBS problem} 
Informally, our method relies on a column generation method in which the ordinary linear programming theory is replaced by the linear lexicographic programming theory. As in any column generation method, the mathematical program to be solved is restricted to a subset of the original variables, which is updated iteratively. We do not depart from this procedure and the following mathematical program will play a central role. Here, $\S'_i$ is a subset of $\S$.

\hspace*{-15cm}
\LPblocktag{P'}{\label{partial-problem}}%
\begin{minipage}{\linewidth-5cm}
	\begin{flalign}\notag
\lexmax\quad & C\x \notag\\
\mbox{s.t.}   \quad  &  \ds{\sum_{s\in \S'_i}x_{is}=1} && \forall i\in I &%\label{eq:assignement}
\\
& \ds{\sum_{i\in I}\sum_{s \in \S'_i \colon s \ni p} x_{is}=1} && \forall p\in \P &%\label{eq:partition} 
\\
& x_{is}\in\{0,1\} && \forall i\in I, s\in \S'_i.& \notag%\label{blabla}  
	\end{flalign}~
\end{minipage}

The overall algorithm is Algorithm~\ref{algo}. It makes use of the notion of ``linear relaxation'' and of ``reduced costs'' in linear lexicographic programming. Reduced costs have already been defined in Section~\ref{subsec:prel}. As expected, the linear relaxation of~\eqref{partial-problem} consists simply in replacing $x_{is}\in\{0,1\}$ by $x_{is}\in[0,1]$.

Apart from relying on linear lexicographic programming, it follows a quite standard strategy for solving integer programs with a huge number of variables and with ``partitioning'' constraints; see, e.g., the work by \cite{Valentina2016}.

\begin{algorithm}
\caption{Exact algorithm for solving~\eqref{main-problem}\label{algo}}
\begin{algorithmic}[1]
    \State\label{init} \hspace{-0.85mm}\textbf{initialize} the $\S'_i$ in such a way \eqref{partial-problem} is feasible;
    \Repeat
    \State\label{lin-rel} compute a primal-dual feasible basis of the linear relaxation of~\eqref{partial-problem};
    %solve the linear relaxation of~\eqref{partial-problem};
    
    \Comment{\begin{footnotesize}denote the optimal value by $\u$\end{footnotesize}}
    \State\label{pricing} for each pilot $i$, find some feasible schedules with lexicographically positive reduced costs and add them to $\S'_i$;
    \Until{\label{until}(no feasible schedule with lexicographically positive reduced cost exists)}
    \State\label{solver} solve~\eqref{partial-problem}; \Comment{\begin{footnotesize}denote the optimal value by $\eell$\end{footnotesize}}
    \State\label{pricing2} for each pilot $i$, find all feasible schedules with reduced cost lexicographically non-smaller than $\eell - \u$ and add them to $\S'_i$;
    \State\label{solver2} solve~\eqref{partial-problem};
    \State \Return an optimal solution of the last resolution;
\end{algorithmic}
\end{algorithm}

%We will end the subsection with further detail on how each step of the overall algorithm can be performed. We explain now why the algorithm solves the problem to optimality.

Step~\ref{lin-rel} is well-defined by Corollary~\ref{cor:dual-basic} because~\eqref{partial-problem} is lexicographically bounded from above. The algorithm terminates: at each iteration of the loop, the size of $\S'_i$ increases for at least one pilot $i$ (no feasible schedule in $\S'_i$ can have a lexicographically positive reduced cost because the basis is primal-dual feasible), and $\S'_i$ is included in $\S$ which is finite. Lemma~\ref{lem:opt-basis} ensures that, when we leave the loop at step~\ref{until}, the vector $\u$ is the optimal value of the linear relaxation of~\eqref{main-problem}. At step~\ref{solver}, the vector $\eell$ is a lexicographic lower bound on the optimal value since any feasible solution of~\eqref{partial-problem} is a feasible solution of~\eqref{main-problem}. Lemma~\ref{gap} shows then that the returned solution is indeed optimal.

We end this subsection with further detail on how each step of the overall algorithm can be performed. 
Step~\ref{init} can be implemented in various ways: heuristics, previous solutions available to the airline, column generation, etc. In the experiments, we will assume that we are given a feasible solution, i.e., simply a collection of $m$ feasible schedules forming a partition of $\P$. %It is a reasonable assumption since usually the PBS follows the rostering process that outputs such a solution. 
Step~\ref{lin-rel} corresponds in the column generation terminology to the resolution of the master problem. The simplex algorithm by Isermann would be relevant here; yet, we follow Akg\"ul's approach avoiding the tricky implementation of a simplex algorithm and relying on the performance of the current linear programming off-the-shelf solvers. It is the object of Section~\ref{sec:master}. Steps~\ref{pricing} and~\ref{pricing2} correspond in the column generation terminology to the resolution of the pricing problem. Here, it requires solving a resource-constrained longest path problem on a DAG with a lexicographic objective. We propose in Section~\ref{sec:pricing} an efficient algorithm for solving such a problem (we have not found any work dealing with this problem in the literature). Steps~\ref{solver} and~\ref{solver2} consist in solving an integer linear lexicographic program. This could be achieved by the standard branch-and-bound method. Nevertheless, it turns out that off-the-shelf solvers already offer efficient tools to solve such programs, and we rely on such a solver in our experiments.

\section{Solving the master problem}\label{sec:master}

The purpose of this section is to explain how to compute a primal-dual feasible basis of \eqref{lin-lex-prog}, which we assume to be feasible and lexicographically bounded from above. It is the task to be performed at step~\ref{lin-rel} of Algorithm~\ref{algo}.

Isermann has sketched an approach for computing such a basis via a ``simplex algorithm for linear lexicographic programs.'' We formalize Akg\"ul's idea, which allows instead to rely on the usual simplex algorithm and to use off-the-shelf solvers.
While Akg\"ul explained that this approach computes an optimal basis, we prove that it actually computes a primal-dual feasible basis. %The key element in our proof is Lemma~\ref{lem:powerful}, which states a 

Let $\cc^l$ be the (column) vector $(c_{l1},c_{l2},\ldots,c_{ln})^\top$. (It is the $l$-th cost vector and its entries are those of the $l$-th row of $C$.) The approach consists in considering the following sequence of linear programs, indexed by $l=1,\ldots,m$, and whose variables are indexed by elements in a set $S^{(l)}$, defined inductively hereafter:

\begin{equation}\tag{P$_{l}$}\label{Pell}
\begin{array}{rl}
\max & \cc_{S^{(l)}}^{l}\cdot\y\smallskip \\
\mbox{s.t.} & A_{S^{(l)}}\y=\boldsymbol{b} \\
& \y\geq\zero\, .
\end{array}
\end{equation}
Set $S^{(1)}=[n]$. Then, for $l=1,\ldots,m$, 
\begin{itemize}
\item define $B^{(l)}$ as any primal-dual feasible basis of (\hyperref[Pell]{P$_{l}$}).
\item set $S^{(l+1)}=\{j\in S^{(l)}\colon c_{lj} = \cc_{B^{(l)}}^{l\top} A_{B^{(l)}}^{-1}\a_j\}$. (This set is formed by the indices of the variables with a reduced cost equal to $0$ with respect to $B^{(l)}$.)
\end{itemize}
We have in particular $B^{(l)}\subseteq S^{(l+1)}\subseteq S^{(l)}$ for all $l\leq m$.

Note that computing this sequence of primal-dual bases can be easily performed with a standard linear programming solver.

\begin{theorem}\label{thm:lex-reg}
The basis $B^{(m)}$ is a primal-dual feasible basis of~\eqref{lin-lex-prog}.
\end{theorem}

The proof of Theorem~\ref{thm:lex-reg} requires an easy technical lemma. We are not aware of any reference in the literature.%For ordinary linear programming, 

\begin{lemma}\label{lem:powerful}
Consider a linear program $\max\{\cc'\cdot\x\colon A'\x =\boldsymbol{b}', \x\geq\zero\}$ with $A'$ being full row rank. Let $B$ be a feasible basis. Set $S=\{j\colon c'_j = \cc_B^{'\top} A_B^{'-1} \a'_j\}$. Let $B'$ be any feasible basis included in $S$. Then $\cc_{B'}^{'\top} A_{B'}^{'-1}= \cc_B^{'\top} A_B^{'-1} $.
%Then the reduced costs with respect to $B$ and those with respect to $B'$ are equal.
\end{lemma}

\begin{proof}
We have $c'_j = \cc_B^{'\top} A_B^{'-1} \a'_j$ for all $j\in S$. Thus, $\cc_{B'}^{'\top} = \cc_B^{'\top} A_B^{'-1} A'_{B'}$, since $B'\subseteq S$, which can be written as $\cc_{B'}^{'\top} A_{B'}^{'-1}= \cc_B^{'\top} A_B^{'-1} $.
\end{proof}

\begin{proof}[Proof of Theorem~\ref{thm:lex-reg}.]
Let $j\in [n]$. Denote by $\overline\cc_j = (\overline c_{1j}, \overline c_{2j}, \ldots, \overline c_{mj})^\top$ the $j$-th reduced cost of \eqref{lin-lex-prog} with respect to $B^{(m)}$. We have $B^{(m)} \subseteq S^{(m+1)} \subseteq S^{(m)} \subseteq S^{(m-1)} \subseteq \cdots \subseteq S^{(l+1)}$, and thus  Lemma~\ref{lem:powerful} with 
\[
\cc'=\cc_{S^{(l)}}^{l}, \quad A'=A_{S^{(l)}}, \quad \boldsymbol{b}' = \boldsymbol{b}, \quad B=B^{(l)}, \quad S=S^{(l+1)}, \quad \text{and} \quad B'=B^{(m)}\, ,
\]
shows that $\cc_{B^{(m)}}^{l\top} A_{B^{(m)}}^{-1} =  \cc_{B^{(l)}}^{l\top} A_{B^{(l)}}^{-1}$. This implies that  $\overline c_{lj} = c_{lj} - \cc_{B^{(l)}}^{l\top} A_{B^{(l)}}^{-1}\a_j$ for all $l\in[m]$.

If $\overline c_{lj} = 0$ for all $l$, we have $\overline\cc_j\leqlex \zero$, as required. We assume thus that there is at least one $l$ for which this equality does not hold. Define $l^*$ as the smallest $l$ such that $\overline c_{lj} \neq 0$.

We prove by induction on $l$ that $j\in S^{(l)}$ for all $1\leq l \leq l^*$.
This is true by definition for $l=1$. Suppose now that it is true for $1 \leq l < l^*$. Since $\overline c_{lj} = c_{lj} - \cc_{B^{(l)}}^{l\top} A_{B^{(l)}}^{-1}\a_j$ and $\overline c_{lj} = 0$ (by $l<l^*$), we have $j\in S^{(l+1)}$, which concludes the induction.

We also have $\overline c_{l^*j} = c_{l^*j} - \cc_{B^{(l^*)}}^{{l^*}\top} A_{B^{({l^*})}}^{-1}\a_j$. This latter quantity is non-positive because $B^{(l^*)}$ is a primal-dual feasible basis of (\hyperref[Pell]{P$_{l^*}$}) and $j\in S^{(l^*)}$, as proved just before. Since $\overline c_{l^*j}\neq 0$, we have therefore $\overline c_{l^*j}<0$, as required.
\end{proof}

\begin{remark}\label{rem:reduced}
Denote by $\overline\cc_j$ the $j$-th reduced cost of \eqref{lin-lex-prog} with respect to $B^{(m)}$. As seen in the proof of Theorem~\ref{thm:lex-reg}, we have $\overline c_{lj} = c_{lj} - \cc_{B^{(l)}}^{l\top} A_{B^{(l)}}^{-1}\a_j$. This means that if we have access to the (row) vectors $\cc_{B^{(l)}}^{l\top} A_{B^{(l)}}^{-1}$ for every $l$, we can easily compute the reduced cost associated to any $j\in[n]$. Standard solvers usually give access for free to this vector $\cc_{B^{(l)}}^{l\top} A_{B^{(l)}}^{-1}$ once \eqref{Pell} has been solved to optimality. Indeed, these are the optimal dual solutions. Actually, we need slightly more than optimality since a primal-dual feasible basis is required, but most solvers provide this possibility. (The simplex algorithm terminates with such a basis.) This is used in our implementation to build efficiently the pricing problem.
\end{remark}

\begin{remark}
The practical computation of $S^{(l)}$ is tricky because it relies on the identification of components equal to $0$, and is therefore subject to numerical instability. This must be carefully addressed by performing all equality tests with a nonzero ``epsilon'' margin.
\end{remark}

\section{Solving the pricing problem}\label{sec:pricing}

The pricing problem, which corresponds to steps~\ref{pricing} and~\ref{pricing2} of Algorithm~\ref{algo}, consists in finding for each pilot $i$ feasible schedules $s\in \S$, if any, with lexicographically positive reduced costs $\overline \cc_{is}$. We formulate therefore the pricing problem for pilot $i$ as
\begin{equation}\label{eq:pricingpb}
\displaystyle \lexmax_{s\in \S}\limits \overline\cc_{is} \, .
\end{equation}

With the general notation of Section~\ref{subsec:prel}, we express the reduced cost $\overline \cc_{is}$ as $\cc_{is} - C_BA_B^{-1}\a_{is}$, where $B$ is a feasible basis of the linear relaxation of \eqref{main-problem}, and where $A$ is the constraint matrix of this problem (and $\a_{is}$ is its column associated to the pair $(i,s)$). Hence, this reduced cost can be written under the form
\begin{equation}
    \label{eq:pricingSubproblemObjective}
\overline \cc_{is} = 
\underbrace{
\begin{bmatrix}
0 \\[-1ex]
\vdots
\\
0
\\[-1ex]
c_{is}
\\
0
\\[-1ex]
\vdots
\\
0
\end{bmatrix}
}_{\substack{\text{specific to the} \\ \text{pricing problem} \\ \text{of pilot $i$;} \\ \text{only objective $i$} \\ \text{is non-zero} } }
-
\underbrace{
\begin{bmatrix}
\lambda_{1i} \\[-1ex]
\vdots \\[-1ex]
\lambda_{mi}
\end{bmatrix}
}_{\substack{\text{not depending on $s$}}}
-
\underbrace{
\sum_{p\in s}
\begin{bmatrix}
\mu_{1p} \\[-1ex]
\vdots \\[-1ex]
 \mu_{mp}
\end{bmatrix}
}_{
\substack{
\text{shared by the pricing} \\ \text{problems of all pilots}
}
}
\, ,
\end{equation}
where the first vector in the right-hand term is $\cc_{is}$ and has its only non-zero entry located on the $i$-th row. The real numbers $\lambda_{li}$ and $\mu_{lp}$ are read from $C_BA_B^{-1}\a_{is}$. (Note that the columns of $C_BA_B^{-1}$ are indexed by $I \cup \P$ and its rows by $I$.)  They can actually be obtained from optimal solutions of some dual problems, but we avoid referring to dual notions and build the pricing problem from the reduced cost expression.

%\comAP{New paragraph starts here}
Pricing problems of column generation approaches for vehicle and crew scheduling problems are typically modeled as resource constrained path problems, and solved with adequate enumeration algorithms.
The main specificity of our pricing problem is that the objective is lexicographic. 
Section~\ref{sec:RCLPP-Lex} introduces the lexicographic resource-constrained longest path problem, and shows that the usual enumeration algorithms for resource-constrained shortest path problems can be extended to that setting.
Section~\ref{sec:modelingAsRCLPP-Lex} then models the pricing problem~\eqref{eq:pricingpb} within this framework.
%This general approach could be followed to address the pricing subproblem of any lexicographic vehicle or crew scheduling problem.

The second specificity of our problem comes from the fact that the $i$-th lexicographic objective of the master problem involves only variables corresponding to pilot $i$. In an interpretation of the problem~\eqref{main-problem} as a Dantzig--Wolfe decomposition, the $i$-th objective involves only variables from the $i$-th block. The consequence, which is highlighted in equation~\eqref{eq:pricingSubproblemObjective}, 
is that the objectives of the pricing problem of pilot $i$ are identical to those of all other pilots (the $\lambda_{li}$ are constant with respect to the pricing problem), except for the $i$-th one.
Section~\ref{sec:simultaneousPricing} shows how a preliminary call to the lexicographic longest path algorithm---which we call the ``reduction'' trick---can exploit this property to solve often the pricing problem for many pilots simultaneously.
%that we can therefore deduce the optimal solution of the majority of the pricing subproblems by screening the $K \in \bbZ_+$ best solutions of the lexicographic resource-constrained longest-path problem that considers only the part of the objective that is shared by all pilots.
We will see in the numerical experiments that this trick strongly improves the practical performance of the algorithms.
%It could be applied to any lexicographic problem solved by column generation and such that the $i^{\mathrm{th}}$ objective involves only variables from the $i^{\mathrm{th}}$ block of the Dantzig-Wolfe decomposition.

The rest of the section is technical and can be skipped by a reader that is not interested in the details of the pricing algorithm implementation.

% It turns out that the pricing problem can be modeled as a special case of a resource-constrained longest path problem with lexicographic costs that we introduce now in a general setting.
%\inlAP{I would add a few general strategy sentences here}

\subsection{The lexicographic resource-constrained longest path problem (RCLPP-Lex)} \label{sec:RCLPP-Lex}

\subsubsection*{The problem} We are given a directed graph $D=(V,A)$ with two particular vertices $o$ and $d$, and a partial ordered set $(R,\preceq)$ of \emph{resources}, which is supposed to have a unique maximal element, denoted by $\uun$. Each arc $a$ is equipped with a non-decreasing map $f_a\colon R\rightarrow R$. Moreover, we are given a non-increasing map $c\colon R\rightarrow \bar\R^m$. %(where $\bar\R$ is the set $\R\cup\{-\infty,+\infty\}$). 
The \emph{resource of a path} $P$, denoted by $r_P$, is defined inductively by $r_P=f_a(r_{P\setminus a})$, where $a$ is the last arc of $P$. The map $f_a$ has to be thought as describing the way the resource of a path is modified when an arc $a$ is added to the path. The resource of a path reduced to a single vertex depends on that vertex and is given as an input of the problem. A path $P$ is \emph{feasible} if $r_P\neq\uun$. We define the lexicographic resource-constrained longest-path problem (RCLPP-Lex) as the problem of computing a feasible $o$-$d$ path $P$ with maximal $c(r_P)$.

A natural way to solve this problem consists in enumerating feasible $o$-$v$ paths and in discarding as many of them as possible by dominance ($f$ is non-decreasing). This might be actually time-consuming, but in many cases longest-path (or shortest-path) problems enjoy extra properties on which algorithms can rely to improve their computational times. We describe now such a property. As we will see, the pricing problem enjoys this property. 

\subsubsection*{When bounds can be used} Suppose that we are given a non-decreasing map $g_a\colon R\rightarrow R$ for each arc $a$ (``reverse-extension'' map) and another component-wise non-decreasing map $h\colon R\times R\rightarrow R$ (``merge'' map). For a path $Q$, we define inductively $r'_Q$ by $r'_Q=g_a(r'_{Q\setminus a})$, where $a$ is the first arc of $Q$. 
When $Q$ is reduced to a single vertex, the definition of $r'_Q$ depends on the vertex and the nature of the problem. Suppose moreover that for every vertex $v$, every $o$-$v$ path $P$ and every $v$-$d$ path $Q$, we have $h(r_P,r'_Q) \preceq r_{P+Q}$. Then, any $b_v$ satisfying $b_v \preceq r'_Q$ for every $v$-$d$ path $Q$ is a bound that can be used to discard $o$-$v$ paths: indeed, if an $o$-$v$ path $P$ has $c(h(r_P,b_v))$ smaller than the current best solution, we know that $P$ cannot be part of an optimal solution.

A {\em meet-semilattice} is a partial ordered set in which every finite non-empty subset $X$ has a (unique) greatest lower bound $\wedge X$ (called {\em meet}). When $(R,\preceq)$ is a meet-semilattice and $D$ is acyclic, any solution $(b_v)_{v\in V}$ of the following dynamic programming equation provides bounds for the problem:
$$ \left\{\begin{array}{ll}
 b_d=  r'_{\text{path reduced to $d$}}  \\[2ex]
 b_v=\bigmeet\limits_{(v,w) \in \delta^{+}(v)}g_{(v,w)}(b_w) \, .
\end{array}\right. $$
(For a proof of this fact, see \citep{parmentierAlgorithmsNonlinearStochastic2019}, or \citep{poulletShiftPlanningDelay2020} where a notation closer to that of the present paper is used.)

\subsubsection*{An enumeration algorithm for acyclic graphs, when bounds are available} Assuming that the bounds $b_v$ are available (e.g., they have been pre-computed), we propose the following algorithm to solve RCLPP-Lex when the graph $D$ is acyclic. In our case, this assumption on the graph will be satisfied. As far as we know, RCLPP-Lex, or close problems, have not been studied in the literature. Nevertheless, the algorithm we propose follows a classical dominance/bound scheme~\citep{irnich2005shortest}.  %It is an extension of \todo{cite Axel's work, or any relevant work} to lexicographic costs.

\begin{algorithm}
	\caption{Algorithm for RCLPP-Lex on DAGs when bounds are available}\label{algo:shortestpath}
	\begin{algorithmic}[1]
		\Require an acyclic digraph $D=(V,A)$, bounds $(b_v)_{v \in V}$;
		\State $\text{\texttt{current\_best\_cost}} \leftarrow (-\infty,\ldots,-\infty)$, $\text{\texttt{current\_best\_solution}} \leftarrow \texttt{undefined}$, and $L_v \leftarrow \varnothing$ for all $v \in V$;
		%\If{$h(r_o,b_o)= \hat{\boldsymbol{1}}$}
		%\State\Return ``There is no feasible $o$-$d$ path in $D$.''
		%\EndIf
		\State $L_o \leftarrow \{\text{path reduced to $o$}\}$;\label{algolp:initial}
        \While{there exists $u$ such that $L_{u} \neq \varnothing$}
         \State $P \leftarrow $ a path in $L_{u}$;\label{algolp:key}
         \State $L_{u} \leftarrow L_u \setminus \{P\}$;
         \ForAll{$a \in \delta^{+}(u)$}\label{algolp:forloop}
         \If{$f_{a}(r_P)\neq \hat{\boldsymbol{1}}$}
         \State $Q \leftarrow P+a$;
         \State $w \leftarrow $ last vertex of $Q$;
         \If{$w=d$} 
         \If{$c(r_Q) \geqslex \text{\texttt{current\_best\_cost}}$}\label{algolp:iftest}
         \State $\text{\texttt{current\_best\_cost}} \leftarrow c(r_Q)$; \label{algolp:bestvalue}
         \State $\text{\texttt{current\_best\_solution}} \leftarrow Q$; \label{algolp:endif}
         \EndIf
         \Else
         \If{($c(h(r_Q,b_w)) \geqslex \text{\texttt{current\_best\_cost}}$) and ($r_Q \not\succcurlyeq r_{Q'}$ for every $Q'\in L_w$)} \label{algolp:elimination}
         \State remove from $L_w$ all paths $Q'$ such that $r_Q \preceq r_{Q'}$; \label{algolp:dominance}
         \State $L_w \leftarrow L_w \cup \{Q\}$;
         \EndIf
         \EndIf
         \EndIf
         \EndFor
         \EndWhile
         \If{$\text{\texttt{current\_best\_solution}}\neq \texttt{undefined}$} \label{algolp:return1}
         \State\Return $\text{\texttt{current\_best\_solution}}$;
         \Else
         \State\Return ``There is no feasible path.''\label{algolp:return2}
         \EndIf
	\end{algorithmic}
\end{algorithm}

The algorithm consists in enumerating implicitly all possible paths originating at $o$: when an $o$-$u$ path is considered, then all possible ways of extending this path with an arc starting at $u$ are tried. As soon as a first feasible $o$-$d$ path has been found, $\text{\texttt{current\_best\_solution}}$ is a best solution identified so far, and $\text{\texttt{current\_best\_cost}} \in \mathbb{R}^{m}$ is its cost. Initially, $\text{\texttt{current\_best\_cost}}$ is equal to $(-\infty,\ldots,-\infty)$, with the convention that it is the objective value of a non-feasible solution, and therefore this variable is always a lower bound on the optimal value.

As discussed above, this enumeration is implicit because the bound is used to discard $o$-$u$ paths that cannot be part of an optimal solution. The list $L_u$ contains only feasible $o$-$u$ paths that have successfully passed this test, which is the first inequality of step~\ref{algolp:elimination}.

In the version shown in Algorithm~\ref{algo:shortestpath}, the list $L_u$ is also cleaned of all ``dominated'' $o$-$u$ paths. This is the second test at step~\ref{algolp:elimination}. More precisely, at any moment, no two paths $Q$ and $Q'$ with $r_Q \preceq r_{Q'}$ are present simultaneously in $L_u$. Indeed, removing such a path $Q'$ will not affect the correctness of the algorithm. Yet, when several optimal paths are sought, this might lead to inexact solutions, and therefore this property of the list $L_u$ must not be kept: dominated paths must be allowed and the second test at step~\ref{algolp:elimination} and step~\ref{algolp:dominance} must be omitted. This is actually the case in our experiments because we add several columns at each iteration of the column generation (steps~\ref{pricing} and~\ref{pricing2} of Algorithm~\ref{algo}).

\subsubsection*{Possible improvements}
There are several places in Algorithm~\ref{algo:shortestpath} where computational time can be spared.

A preliminary test for feasibility, using $b_o$, can be performed at step~\ref{algolp:initial}: if $h(r_{\{o\}},b_o)$ is equal to $(-\infty, \ldots,-\infty)$ (where $\{o\}$ stands for the path reduced to $o$), then it is not possible to get a better cost (and in general this will coincide with the absence of any feasible solution).

A topological order on the vertices can be used: at step~\ref{algolp:forloop}, choose $a$ so that its head is the highest according to this order. Indeed, $o$-$d$ paths will be computed earlier, which increases the chances of improving the current best solution, and hence the number of discarded paths.

A ``key'' attached to the paths can be used to prioritize the order in which the elements in $L_u$ must be considered: define it as $\operatorname{key}(P)\coloneqq h(r_P,b_v)$, where $v$ is the last vertex of $P$; then
    \begin{itemize}
        \item at line~\ref{algolp:key}, choose $P\in\bigcup_{u\in V}L_u$ maximizing $\operatorname{key}(P)$. Indeed, this choice increases the chances of constructing paths with larger costs.
        \item before entering the {\bf for} loop of line~\ref{algolp:forloop}, test whether $\operatorname{key}(P)\leqlex \text{\texttt{current\_best\_cost}}$ for the $P$ maximizing $\operatorname{key}(P)$, in which case we can stop.
    \end{itemize}

These improvements have been implemented in the code used in the experiments.

\subsubsection*{Solving variants}
The variant of the pricing problem that is solved in step~\ref{pricing} of Algorithm~\ref{algo} consists in generating a number of paths, say $N$ (if any), with the best lexicographically positive reduced costs. Algorithm~\ref{algo:shortestpath} can be adapted to solve this variant, which we call ``$N$-best-paths.'' To do so, we use a plural $\text{\texttt{current\_best\_solutions}}$ to denote the list of the best feasible $o$-$d$ paths identified so far. The smallest cost of these paths is denoted by $\text{\texttt{current\_smallest\_cost}} \in \R^m$ (initially equals $(+\infty,\ldots,+\infty)$). Then, we replace steps~\ref{algolp:iftest} to~\ref{algolp:endif} by

\begin{algorithm}[H]
\begin{algorithmic}
\If{$|\text{\texttt{current\_best\_solutions}}| < N$}
\State $\text{\texttt{current\_best\_solutions}} \leftarrow \text{\texttt{current\_best\_solutions}} \cup \{Q\}$;
\If{$c(r_Q) \leqslex  \text{\texttt{current\_smallest\_cost}}$}
\State $\text{\texttt{current\_smallest\_cost}}\leftarrow c(r_Q)$;
\EndIf
\Else
\If{$c(r_Q) \geqslex \text{\texttt{current\_smallest\_cost}}$}
\State Delete from $\text{\texttt{current\_best\_solutions}}$ a path with cost $\text{\texttt{current\_smallest\_cost}}$;
\State $\text{\texttt{current\_best\_solutions}} \leftarrow \text{\texttt{current\_best\_solutions}} \cup \{Q\}$;
\State $\text{\texttt{current\_smallest\_cost}}\leftarrow \min\{c(r_P)\colon P\in \text{\texttt{current\_best\_solutions}}\}$;
\EndIf
\EndIf
\end{algorithmic}
\end{algorithm}

\noindent and step~\ref{algolp:elimination} by

\begin{algorithm}[H]
\begin{algorithmic}
\If{$|\text{\texttt{current\_best\_solutions}}|<N$ or $c(h(r_Q, b_w)) \geqslex \text{\texttt{current\_smallest\_cost}} $}
\EndIf
\end{algorithmic}
\end{algorithm}

\noindent Finally, we omit the dominance check of step~\ref{algolp:dominance}, and adapt steps~\ref{algolp:return1} to~\ref{algolp:return2} for this variant. 

Similarly, Algorithm~\ref{algo:shortestpath} can be adapted to solve the variant of the pricing problem of step~\ref{pricing2} of Algorithm~\ref{algo}. For this variant, we also use a list $\text{\texttt{current\_best\_solutions}}$ of selected $o$-$d$ paths, we replace step~\ref{algolp:iftest} by

\begin{algorithm}[H]
\begin{algorithmic}
\If{$c(r_Q) \geqlex \eell-\u$}
\EndIf
\end{algorithmic}
\end{algorithm}

\noindent and step~\ref{algolp:elimination} by

\begin{algorithm}[H]
\begin{algorithmic}
\If{$h(r_Q, b_w) \lexgeq \eell - \u$}
\EndIf
\end{algorithmic}
\end{algorithm}

\noindent we omit steps~\ref{algolp:bestvalue} and \ref{algolp:dominance}, and we  adapt steps~\ref{algolp:return1} to~\ref{algolp:return2} for this variant.

\subsection{Modeling and solving the pricing problem as an RCLPP-Lex}
\label{sec:modelingAsRCLPP-Lex}

\subsubsection*{Modeling as an RCLPP-Lex} We explain now how to model the pricing problem with the rules given in Section~\ref{subsec:pb} as an RCLPP-Lex (see definition in Section~\ref{sec:RCLPP-Lex}). To each pilot $i$, we associate a directed acyclic graph $D_i=(V_i,A_i)$. The set of vertices $V_i$ is defined as $\P \cup \{o,d\}$, where the origin $o$ and the destination $d$ represent respectively the beginning and the ending of the month. The arc set $A_i$ is formed by arcs $(o,p)$ and $(p,d)$ for each $p \in \P$, and by arcs $(p,p')$ for every pairings $p,p'\in \P$ such that $p'$ can be operated after $p$. %i.e.,  $e_u \leq b_v$ and $a_u \leq s_v$, where $e_u$ (resp.\ $a_u$) is the ending time  (resp.\ end of the rest after pairing) of $p_u$, and $s_v$ (resp.\ $b_v$) is the starting time (resp.\ start of the rest before pairing) of $p_v$.

 By construction of $D_i$, each feasible schedule for the pilot $i$ can be represented as an $o$-$d$ path in this graph. However, the opposite is not true since an $o$-$d$ path may not satisfy all the rules for a feasible schedule. One way to take into account these rules is to model them using resources in the graph, as we explain now. We omit the rule about the $55$ working hours for sake of readability. The ideas for taking it into account are not essentially different from those explained for the other rules; yet they are quite technical. Nevertheless, it has been taken into account in the experiments. The rules we deal with are thus (Section~\ref{subsec:pb}): the number of days on, the seven consecutive days off, and the number of flight hours. Note that the conflicts between the pairings have been encoded in the construction of the arcs of $D_i$.
 
We define the set of resources $R=(\Z_+ \times \{0,1\} \times \Z_+ \times \R^{m}) \cup \{\uun\}$. A resource $r$ distinct from $\uun$ is thus a four-tuple $(r^o,r^b,r^f,r^c)$ where the first component $r^o$ corresponds to the number of days on, the second component $r^b$ is binary and indicates whether there are at least seven consecutive days off, the third component $r^f$ corresponds to the number of flight hours, and the fourth component $r^c$ is used to compute $\overline\cc_{is}$. For such elements, we define the partial order $\preceq$ on $R$ by 
\[
(r_1^o,r_1^b,r_1^f,r_1^c)\preceq (r_2^o,r_2^b,r_2^f,r_2^c) \quad \text{if} \quad r_1^o\leq r_2^o, \quad r_1^b\geq r_2^b, \quad r_1^f\leq r_2^f, \quad \text{and}\quad r_1^c\geqlex r_2^c \, .
\]
The element $\uun$ is then defined as being greater than any other element in $R$. Note that $(R,\preceq)$ is a meet-semilattice. The resource $r_P$ of a path $P$ reduced to a single vertex $v \in V_i$ is set to $(0,0,0,\zero)$ if $v=o$ and $\uun$ otherwise. Consider for $a\in A_i$, $r^o\in\Z_+$, $r^b\in \{0,1\}$, $r^f\in\Z_+$, and $r^c\in \R^{m}$ the following functions:
\begin{flalign*}
f_a^o(r^o) &= r^o + (\text{number of days on in the head pairing of $a$}),\\
f_a^b(r^b) &= \left\{\begin{array}{ll} 1 & \text{if there are at least seven consecutive days off} \\ &  \text{between the tail pairing and the head pairing of $a$,} \\[1ex] r^b & \text{otherwise,}\end{array}\right. \\
f_a^f(r^f) &= r^f + (\text{number of flight hours in the head pairing of $a$})\, ,\\
f_a^c(r^c) &= r^c + (\text{cost associated to the head pairing of $a$})\, .
\end{flalign*}
If the head of $a$ is not $d$, the cost associated to the head pairing $p$ is set to \[(\mu_{1p},\ldots,\mu_{(i-1)p},g_{ip}+\mu_{ip},\mu_{(i+1)p},\ldots,\mu_{mp})^\top\, .\]
In the above definitions, $d$ is treated as a fictitious pairing with the following characteristics: its number of days on is $0$ (the starting and ending times are equal to the end of the month), its number of flight hours is $0$, and its cost is $(\lambda_{1i},\ldots,\lambda_{mi})^{\top}$. 

We can now define $f_a\colon R\rightarrow R$ for every arc $a\in A_i$. For $r\in R$, we set
\[
f_a(r) = \left\{\begin{array}{ll}
\rho\big(f_a^o(r^o),f_a^b(r^b),f_a^f(r^f),f_a^c(r^c)\big) & \text{if $r\neq\uun$ and ($\text{head of $a$}\neq d$ or $f_a^b(r^b)=1$),}\\[0.5ex]
\uun & \text{otherwise,} 
\end{array}\right.
\]
where
\[
\rho(x,y,z,t) = \left\{\begin{array}{ll}
(x,y,z,t) & \text{if $x \leq 17$ and $z \leq 85$,} \\[0.5ex]
\uun & \text{otherwise.} \\
\end{array}\right.
\]
The map $f_a$ is non-decreasing.

We define $c\colon R \rightarrow \bar\R^m$ by
\[
c(r) = \left\{ \begin{array}{ll} 
r^c & \text{if $r\neq \uun$,}\\
(-\infty,\ldots,-\infty)^\top & \text{otherwise.}
\end{array}\right.
\]
It is a non-increasing map.

In this way, the problem RCLPP-Lex we get models exactly the problem of finding a feasible schedule for pilot $i$ with a lexicographically maximal reduced cost. Indeed, by definition of $f_a$, the resource $r_P$ of an $o$-$d$ path $P$ in $D_i$ (as defined in Section~\ref{sec:RCLPP-Lex}) is equal to $\uun$ precisely when the path $P$ does not correspond to a feasible schedule (one of the three rules is not satisfied). Moreover, the cost of such a path is equal to $\overline\cc_{is}$ when the path encodes the schedule $s$.

 \subsubsection*{Bounds can be used} The algorithm proposed in Section~\ref{sec:RCLPP-Lex} to solve RCLPP-Lex requires extra properties to be applied, namely that $D_i$ be acyclic, and that ``reverse-extension'' and ``merge'' maps be available. This is the case. We define $g_a\colon R\rightarrow R$ for every arc $a\in A_i$. It is exactly the same definition as $f_a$, except that the vertex $d$ is replaced by the vertex $o$ in the feasibility test, and  when a path $Q$ is reduced to a single vertex, we set $r'_Q$ to $(0,0,0,\zero)$ if that vertex is $d$ and $\uun$ otherwise. For sake of completeness, we write the full definition. We define for every  $r\in R$

\[
g_a(r) = \left\{\begin{array}{ll}
\rho\big(f_a^o(r^o),f_a^b(r^b),f_a^f(r^f),f_a^c(r^c)\big) & \text{if $r\neq\uun$ and ($\text{tail of $a$}\neq o$ or $f_a^b(r^b)=1$),}\\[0.5ex]
\uun & \text{otherwise.} 
\end{array}\right.
\]
The map $g$ is non-decreasing.

 With $r_P$ and $r'_Q$ defined as in Section~\ref{sec:RCLPP-Lex}, we have the following lemma, where $h\colon R\times R \rightarrow R$ is defined by 
\[
h(r,r') = \left\{\begin{array}{ll}
\rho \big(r^o+r'^{o},1,r^f+r'^{f},r^c+r'^{c} \big) & \text{if $r \neq \uun$ and $r'\neq \uun$ and $\max(r^b, r'^{b}) = 1$} \\ 
 \uun & \text{otherwise.}
\end{array}
\right.
\]
%(It is well-defined since with the above definition every element in the image of $c$ has all its components distinct from $-\infty$.)

\begin{lemma}
For every $o$-$v$ path $P$ and every $v$-$d$ path $Q$, we have $h(r_P,r'_Q) \preceq r_{P+Q}$.
\end{lemma}

\begin{proof}
If $r_{P+Q}=\uun$, the inequality is satisfied. We can thus assume that $r_{P+Q} \neq \uun$. By definition of $f_a$, we have $f_a(\uun)=\uun$, and hence $r_P\neq \uun$. By the definitions of $f_a$ and $g_a$, we have $r'_{P+Q}=r_{P+Q}$ and $g_a(\uun)=\uun$, and thus similarly $r'_Q\neq\uun$.

Since $r_{P+Q}$ is different from $\uun$, it takes the form of a four-tuple. It is straightforward to check that the definition of $f_a$ and $g_a$ ensures that we actually have
$r_{P+Q} = (r_P^o+{r'_Q}^{o},\max(r_P^b,{r'_Q}^b),r_P^f+{r'_Q}^{f},r_P^c + {r'_Q}^c)$ (the maps $g_a$ counts the resource along an arc in the same way as $f_a$). The fact that $r_{P+Q}\neq\uun$ implies that for the last arc $a$ of $P+Q$ we have $f_a^b(r_{(P+Q)\setminus a}^b)=1$. It means that there is at least one arc $a'$ of $P+Q$ for which there are at least seven consecutive days off between the tail pairing and the head pairing of $a'$. We have thus $\max(r_P^b,{r'_Q}^b)=1$. Hence, $h(r_P, r'_Q) = \rho(r_P^o+{r'_Q}^{o},1,r_P^f+{r'_Q}^{f},r_P^c + {r'_Q}^c)$. An argument similar as the one for the seven consecutive days off shows that $r_P^o+{r'_Q}^{o}\leq 17$ and $r_P^f+{r'_Q}^{f}\leq 85$. Therefore, $h(r_{P+Q})=r_{P+Q}$, which implies the desired result.
\end{proof}

Since $h$ is componentwise non-decreasing, this lemma shows that Algorithm~\ref{algo:shortestpath} can be applied, with the bounds $b_v$ being computed in a pre-processing step relying on dynamic programming, as explained in Section~\ref{sec:RCLPP-Lex}.

\begin{remark}%\comAP{New remark}
The set of resources $R$ is of dimension $4$, which starts to be rather large for dominance. The number of non-dominated paths tends to be large, which means that few paths are cut by dominance, and large sets of non-dominated paths must be stored. In our situation, the CPU time lost performing tests with these sets will be in general larger than the one gained by discarding paths. As mentioned earlier, we are not using dominance for the  ``N-best-paths'' variant (since otherwise the correctness of the algorithm is not ensured); yet, even for the original version of the problem RCLPP-Lex, where a single optimal path is sought, this indicates that dominance should be avoided when solving the pricing problem.
%We therefore do not use dominance even when we seek only the best paths -- it is not used by default for the  ``N-best-paths'' variant.
\end{remark}

\subsection{The ``reduction'' trick}\label{subsec:Imp-RCLPP-Lex}
\label{sec:simultaneousPricing}

As emphasized by the comments in~\eqref{eq:pricingSubproblemObjective}, the only objective in $\overline \cc_{is}$ that depends on $i$ and that must decomposed along the paths in the pricing problem is the $i$-th objective. We can exploit this structure as follows.
Consider the following minimization problem:
\begin{equation}\label{genericpbj}
\lexmin_{s\in \S}\limits \sum_{p\in s}\begin{bmatrix}
\mu_{1p}	\\
\vdots\\
\mu_{(m-1)p}
\end{bmatrix}
\, .
\end{equation}
Consider an arbitrary positive integer $K$. Denote by $L$ the set of the best $K$ feasible solutions of problem~\eqref{genericpbj}. Moreover, we define $i^*$ as the smallest index $i$ such that at least two solutions in $L$ provide distinct values for $\sum_{p\in s}\mu_{ip}$. If such an index does not exist, we set $i^*\coloneqq m+1$.

\begin{proposition}\label{prop:improve-path}
Every optimal solution of problem~\eqref{eq:pricingpb} for $i>i^*$ is contained in $L$.
\end{proposition}

This proposition can be used to shorten in many cases the computational time for solving problem~\eqref{eq:pricingpb}. The following way has been implemented.  Determine $L$ with the adaptation of Algorithm~\ref{algo:shortestpath} for solving the variant ``$N$-best-paths'' (with $N\coloneqq K$) and compute $i^*$. If $i^*\leq m-1$, we simply scan $L$ to find the optimal solutions of problem~\eqref{eq:pricingpb} for $i>i^*$, and we use only Algorithm~\ref{algo:shortestpath} for $i\leq i^*$. If $i^* \geq m$, then we cannot not conclude anything, and we solve problem~\eqref{eq:pricingpb} for each $i$ with Algorithm~\ref{algo:shortestpath}.

We call it the ``reduction'' trick.

\begin{proof}[Proof of Proposition~\ref{prop:improve-path}.]
Let $s^*$ be an optimal solution of problem~\eqref{eq:pricingpb} for $i>i^*$. Then $s^*$ is optimal for
\begin{equation*}
\lexmax_{s\in \S}\limits \begin{bmatrix}
&-\lambda_{1i}-\sum_{p\in s}\mu_{1p}	\\
&\vdots\\
c_{is}&-\lambda_{ii}-\sum_{p\in s}\mu_{ip}	\\
&\vdots\\
&-\lambda_{mi}-\sum_{p\in s}\mu_{mp}
\end{bmatrix}\, .
\end{equation*}
By definition of the lexicographic optimization, $s^*$ is also an optimal solution to
\begin{equation*}
\lexmax_{s\in \S}\limits \begin{bmatrix}
&-\lambda_{1i}-\sum_{p\in s}\mu_{1p}	\\
&\vdots\\
&-\lambda_{(i-1)i}-\sum_{p\in s}\mu_{(i-1)p}	\\
\end{bmatrix}=\begin{bmatrix}
-\lambda_{1i}	\\
\vdots\\
-\lambda_{(i-1)i}	\\
\end{bmatrix}+ \lexmax_{s\in \S}\limits \begin{bmatrix}
-\sum_{p\in s}\mu_{1p}	\\
\vdots\\
-\sum_{p\in s}\mu_{(i-1)p}	\\
\end{bmatrix} \, .
\end{equation*}
Therefore, $s^*$ is an optimal solution to
\begin{equation} \label{pbi}
\lexmax_{s\in \S}\limits \begin{bmatrix}
-\sum_{p\in s}\mu_{1p}	\\
\vdots\\
-\sum_{p\in s}\mu_{(i-1)p}	\\
\end{bmatrix} \, .
\end{equation}
Since $i-1 \geq i^*$, there exist at least two solutions in $L$ providing different values of $-\sum\limits_{p\in s} \begin{bsmallmatrix}
\mu_{1p}	\\
\vdots\\
\mu_{(i-1)p}	\\
\end{bsmallmatrix}$ (at least at row $i^*$). This implies that $L$ contains all the optimal solutions of problem~\eqref{pbi} including $s^*$, and the result follows.
\end{proof}

\section{Numerical experiments}

\subsection{Instances and setting}

The method proposed in the paper---Algorithm~\ref{algo} with Algorithm~\ref{algo:shortestpath} for the pricing problems, together with the ``reduction'' trick of Section~\ref{subsec:Imp-RCLPP-Lex}---has been tested on eight sets of industrial instances provided by Air France. All the sets have five different instances, except the last set that contains one instance.  The first column of Table~\ref{tab:instances} provides the name of these sets. The next three columns present: the number of instances of each set, the number $m$ of pilots, and the number of pairings. The last two columns give the size of the graph associated with each pilot: the number of vertices and arcs. All the instances correspond to a horizon of one month with long-haul flights (international flights). The number of days on of the pairings varies from 2 to 8 days.

\begin{table}
\caption{Problems sizes}
\centering
\begin{tabular}{lrrrrrr}
\toprule
& \multicolumn{3}{c}{Instance} &
\multicolumn{2}{c}{Graph} \\
\cmidrule(r{4pt}){2-4} \cmidrule(l){5-6}
& \multicolumn{1}{c}{Number of} & \multicolumn{1}{c}{Number of}  & \multicolumn{1}{c}{Number of} &  \multicolumn{1}{c}{Number of} & \multicolumn{1}{c}{Number of} \\
&\multicolumn{1}{c}{instances} & \multicolumn{1}{c}{pilots ($m$)}  & \multicolumn{1}{c}{pairings} &  \multicolumn{1}{c}{vertices} & \multicolumn{1}{c}{arcs}\\
\midrule
$I_1$ &5 & 17& [59; 67] &[61; 69]  &[1,235; 1,646] \\
$I_2$ &5 &25 &[96; 110]  &[98; 112]&[3,206; 4,604]  \\
$I_3$ &5 &50 &[150; 160] & [152; 162] &[7,320; 8,962]   \\
$I_4$ &5 &70 & [209; 210]&[211; 212]  &  [14,114; 15,232]\\
$I_5$ &5 &80 &[239; 240] &[241; 242]  & [18,581; 19,389]\\
$I_6$ &5&90 & [269; 270]& [271; 272]& [23,301; 25,210] \\
$I_7$ &5 &100 & [297; 300] & [299; 302] & [28,687; 29,664] \\
$I_8$ &1 & 150 & 497 &499&79,558\\
\bottomrule
\end{tabular}
\label{tab:instances}
\end{table}

The experiments have been conducted on a server with 192 GB of RAM and 32 cores at 3.30 GHz. The algorithms have not been parallelized. The mathematical models have been solved using Gurobi 9.02, with no ``presolve,'' to ensure a total control on the bases and variables considered by the solver. This is required by the computation of the optimal basis; see Section~\ref{sec:master}. 

Preliminary experiments have shown that the following choices provide good results in general:
\begin{itemize}
    \item For step~\ref{pricing}, we set the number of columns added at each iteration to $10$ columns (if any). (Algorithm~\ref{algo:shortestpath} is used with the adaptation for the variant ``$N$-best-paths'' with  $N \coloneqq 10$.)
\item For the ``reduction'' trick of Section~\ref{subsec:Imp-RCLPP-Lex}, we set $K \coloneqq 500$. 
\end{itemize}

We have realized during the experiments how tricky is the management of the numerical precision within lexicographic optimization: two solutions wrongly ordered according to one objective can lead to a huge degradation on objectives with lower priority; in the pricing problem, this can even lead to stop the column generation while there are still columns that can improve drastically the solution for objectives with lower priority. However, we did not require any new idea in that respect; the only message is that the precision has to be dealt with in a very rigorous manner.

\subsection{Results}
The results are reported in Table~\ref{tab:cpualgo2} in two parts: the average over the five instances of each set, and the worst instance in terms of the computational time. The first six columns of each part present respectively the time (in seconds) needed by the overall algorithm (Algorithm~\ref{algo}), the master problem, the pricing problems (which is formed by the total time spent in Algorithm~\ref{algo:shortestpath} plus the small extra time for the ``reduction'' trick), the integer linear lexicographic program of step~\ref{solver}, the pricing with gap of step~\ref{pricing2}, and the integer linear lexicographic program of step~\ref{solver2}. 

For some instances, the pricing with gap of step~\ref{pricing2} of Algorithm~\ref{algo} generates a huge number of columns which complicates solving the integer linear lexicographic program of step~\ref{solver2}. We limited the number of columns generated at step~\ref{pricing2} to 100,000 columns for instances with at least 70 pilots. The column ``Most sen. p. w/gap'' presents the (average) index of the most senior pilot who is ``unsatisfied,'' i.e., the most senior pilot for which the algorithm has not certified the optimality of his final score. The column ``Gap'' then shows the (average) gap between the upper bound on the score of most senior ``unsatisfied'' pilot, provided by the LP relaxation (denoted by $UB$), and his final score (denoted by $LB$). This gap is evaluated as $(UB-LB)/LB$, converted in percentages. For both columns, the average is taken only on the instances for which optimality of the final solution has not been certified (the number of such instances is given between parentheses).

\begin{table}
\caption{Results of the overall method.}
\centering
\resizebox{\textwidth}{!}{\begin{tabular}{lrrrrrrrrrrrrrrrr}
\toprule
& \multicolumn{8}{c}{Average} &
\multicolumn{8}{c}{Worst} \\
\cmidrule(r{4pt}){2-9} \cmidrule(l){10-17}
& \multicolumn{1}{c}{Total} & \multicolumn{1}{c}{Master}  & \multicolumn{1}{c}{Pricing}  & \multicolumn{1}{c}{First} & \multicolumn{1}{c}{Pricing} & \multicolumn{1}{c}{Second} & \multicolumn{1}{c}{Most sen.} & \multicolumn{1}{c}{Gap} & \multicolumn{1}{c}{Total} & \multicolumn{1}{c}{Master}  & \multicolumn{1}{c}{Pricing}  & \multicolumn{1}{c}{First} & \multicolumn{1}{c}{Pricing} & \multicolumn{1}{c}{Second} & \multicolumn{1}{c}{Most sen.} & \multicolumn{1}{c}{Gap}\\
& \multicolumn{1}{c}{(s)}  &  \multicolumn{1}{c}{pb.} & \multicolumn{1}{c}{pbs.} & \multicolumn{1}{c}{ILLP} & \multicolumn{1}{c}{w/ gap} & \multicolumn{1}{c}{ILLP}& \multicolumn{1}{c}{p. w/ gap} & \multicolumn{1}{c}{($\%$)}  & \multicolumn{1}{c}{(s)} & \multicolumn{1}{c}{pb.} & \multicolumn{1}{c}{pbs.}& \multicolumn{1}{c}{ILLP} & \multicolumn{1}{c}{w/ gap} & \multicolumn{1}{c}{ILLP} & \multicolumn{1}{c}{p. w/ gap} & \multicolumn{1}{c}{($\%$)}\\
&  &  \multicolumn{1}{c}{(s)}   & \multicolumn{1}{c}{(s)}  &\multicolumn{1}{c}{(s)}  & \multicolumn{1}{c}{(s)} & \multicolumn{1}{c}{(s)}  & \multicolumn{1}{c}{(\# non-opt.)} &   & & \multicolumn{1}{c}{(s)} & \multicolumn{1}{c}{(s)} & \multicolumn{1}{c}{(s)} & \multicolumn{1}{c}{(s)} & \multicolumn{1}{c}{(s)} &  & \\
\midrule
$I_1$ & 12.0 & 0.7 & 1.0 & 1.7 & 0.1 & 8.4 & - & -& 44.1 &     0.7 &        0.8 &      2.1 &        0.5 &      40.0 &  - & - \\
$I_2$ & 57.0 & 3.1 & 42.7 & 11.2 & 0.0 & 0.0  & - & -& 160.2 &      3.9 &        142.8 &      13.5 &            0.0 &            0.0  & - & -\\
$I_3$ &250.4 & 21.0 & 121.8 & 107.6 & 0.0 & 0.0  & - & -& 408.8 &      25.3 &      245.1 &      138.3 &            0.0 &            0.0  & - & - \\
$I_4$ & 1,118.4 & 61.3 & 223.1 & 311.0 & 1.6 & 521.5 & 38.3(3) & 62.8 &  1585.8 &      78.7 &      150.6 &      309.6 &        4.3 &      1042.7  & 20 & 13.3\\
$I_5$  & 1,967.8 & 105.1 & 975.2 & 530.8 & 0.5 & 356.2 & 53.0(1) & 8.3&3,802.4 &      111.1 &      1,236.8 &      671.0 &        2.3 &      1,781.2  & 53 & 8.3 \\
$I_6$ &  3,576.5 & 151.1 & 1,591.8 & 916.8 & 3.5 & 913.4 & 41.0(2) & 22.9&6,613.1 &      99.3 &      3,892.3 &      1,278.1 &        2.1 &      1,341.3  & 56 & 33.3\\
$I_7$ &3,223.4 & 262.4 & 604.1 & 993.1 & 2.3 & 1,361.5  & 50.3(3) & 25.6 &4,693.1 &      263.5 &      754.9 &      1,084.0 &        3.9 &      2,586.8 & 71 & 25.0  \\
$I_8$&79,556.0 & 3,011.7 & 28,657.9 & 47,886.4 & 0.0 & 0.0 & - & -&-&-&-&-&-&-&-&-\\
\bottomrule
\end{tabular}}
\label{tab:cpualgo2}
\end{table}

\subsection{Comments}

Table~\ref{tab:cpualgo2} shows that, except $I_8$, the method finishes in less than two hours with good solutions for all instances, which makes it relevant for practical applications. We emphasize again that the overall method is conceptually simple, which is an appealing feature when considering implementation aspects and maintenance of the code.

Further examination of the computational results (not presented in the tables) actually shows that, except $I_8$, the $26$ instances solved to optimality required around $1$ hour.

There are only $9$ out of the $36$ instances in total for which optimality has not been established. Yet, even for those $9$ instances, the solutions are good since the pilots that remain ``unsatisfied'' are among the less senior ones. The values of the gaps have to be mitigated by the fact that they concern only those ``unsatisfied'' pilots, who are not senior. Moreover, the way scores are assigned by the pilots makes that even a small unsatisfaction may lead to huge gaps: they often formulate a few number of preferences, and not satisfying only one of those may have a strong impact on the final score.

The big instance $I_8$ with $150$ pilots has also been solved to optimality. The computational time is much longer---almost one day---but note that in this context longer computational times are acceptable.% (\cite{achour2007} presented computational times of that order of magnitude).

%The average improvement in the total CPU time is greater than $73\%$ on the instances that have been solved by the two algorithms, and reaches $90\%$ on $I_6$. If we consider the linear relaxation, we can observe that not only the CPU time of the pricing problems that has been reduced but also the CPU time of the master problem. This is due to the decrease in the number of column generation iterations and the number of generated columns.

Moreover, our algorithm presents better computational time than the exact approach proposed by \cite{achour2007}. This latter can solve instances of size up to $91$ pilots, and the computational time of one of these instances exceeds five days. However, this comment should be taken with care since the two methods have not been tested on the same instances.

The computational time increases while increasing the size of the instances. The number of pilots has a direct impact on the number of linear programs and pricing problems solved at each iteration of the column generation, whereas increasing the number of pairings increases the size of the graphs and the set of feasible schedules. 

We note that, even if this column generation is applied in a context that departs from the usual single objective framework, the distribution of the computational time between the master problem and the pricing problem follows the usual behavior, namely that most of this time is spent within the pricing problem (82.28\%). Any neat improvement of the lexicographic longest path algorithm would therefore be beneficial to the overall method. 

%The results also show that a high proportion of the computational time of the linear relaxation ($82.46\%$ on average) is devoted to solving the master problem, which departs from the standard behavior of column generation algorithms. This can be explained by the fact that each iteration requires solving $m$ linear programs (and not a single one as usual) and that the ``presolve'' of Gurobi has been disabled. 

\subsection{Further statistics on the method}

Table~\ref{tab:CGalgo2} provides more details about the columns added during the resolution of the relaxation of problem~\eqref{partial-problem} (essentially step~\ref{pricing} of Algorithm~\ref{algo}).

\begin{table}
\caption{Statistics on the column generation.}
\scriptsize
\begin{tabular}{lrrrrrrr}
\toprule
&  & & & \multicolumn{2}{c}{Pricing problem~\eqref{genericpbj}}& \multicolumn{2}{c}{Pricing problem~\eqref{eq:pricingpb}}\\
\cmidrule(r{4pt}){5-6} \cmidrule(l){7-8}
  &  \multicolumn{1}{c}{Avg. eliminated}& \multicolumn{1}{c}{Column generation}& \multicolumn{1}{c}{Generated}& \multicolumn{1}{c}{Saved} & \multicolumn{1}{c}{Cuts} & \multicolumn{1}{c}{Saved}& \multicolumn{1}{c}{Cuts}\\
 
 &  \multicolumn{1}{c}{subproblems (\%)}& \multicolumn{1}{c}{iterations}& \multicolumn{1}{c}{columns}& \multicolumn{1}{c}{paths} & \multicolumn{1}{c}{by LB}& \multicolumn{1}{c}{paths}& \multicolumn{1}{c}{by LB}\\
\midrule
$I_1$ & 64.4& 33.4 & 4,289.8&208,555.6&	448,456.2&	332,792.6&	1,729,496.0\\
$I_2$ & 60.0 & 46.4 & 9,391.4&1,604,777.8&	10,363,891.0&	14,023,259.6&	108,830,561.4\\
$I_3$ & 52.0 & 70.6 & 24,256.2&3,303,615.8&	13,759,397.2&	48,637,332.2&	232,219,561.2\\
$I_4$ &56.8 & 87.4 & 44,155.8&4,336,915.6&28,876,875.2&
50,113,647.2&500,247,168.2 \\
$I_5$ &46.4 & 98.4 & 56,354.6&10,560,068.2	&90,417,732.0&	188,055,428.4&	1,530,782,977.0\\
$I_6$ & 51.9 & 102.6 & 67,087.8&14,257,280.8	&124,079,553.8&	263,535,659.4	&995,461,637.4  \\
$I_7$ & 64.3 & 124.6 & 94,865.4&4,218,223.8&	35,555,643.8&	100,001,567.4&	1,051,362,117.0 \\
$I_8$&70.8 & 286.0 & 335,159.0 &48,131,858.0&1,071,822,361.0&1,264,065,910.0&727,950,919.0\\
\bottomrule
\end{tabular}
\label{tab:CGalgo2}
\end{table}

``Avg. eliminated subproblems (\%)'' shows the average percentage of problems~\eqref{eq:pricingpb} per column generation iteration that the ``reduction'' trick solves by a single call to Algorithm~\ref{algo:shortestpath}. Without this trick, the number of such problems to be solved at each iteration is $m$, the number of pilots. As we explain in Section~\ref{subsec:improve-reduc}, the trick clearly improves the computational times spent in the pricing problems (e.g., reduction of this time by 50\% in the case of $I_8$). Complementary experiments have shown that choosing much larger $K$ does not significantly reduce the number of pricing problems to be solved, while it increases the computation time of the ``$N$-best-paths'' variant of Algorithm~\ref{algo:shortestpath}.

%As shown by complementary experiments (see Section~\ref{subsec:improve-reduc}), this might lead to prohibitive computational times. With the above choice of $K$, this trick is more efficient when the number of pilots is . Yet, choosing a larger $K$ might lead to prohibitive computational time in Algorithm~\ref{algo:shortestpath}. (For some instances with many equivalent schedules---in terms of scores---experiments have shown that $K$ must be considerably larger to see an increase of the number in this column.) %Over all instances, $67\%$ of those problems are solved by this trick.

``Column generation iterations'' is the number of times the {\bf repeat}-{\bf until} loop in Algorithm~\ref{algo} is repeated.

``Generated columns'' is the total number of columns that have been added by step~\ref{pricing} of Algorithm~\ref{algo}.

%Since two different pricing problems are solved in the column generation, 
The statistics of Algorithm~\ref{algo:shortestpath} are distinguished according to whether it is used to solve the pricing problem~\eqref{genericpbj} or the pricing problem~\eqref{eq:pricingpb} at step~\ref{pricing} of Algorithm~\ref{algo}.
 The columns ``Saved paths'' and ``Cuts by LB'' present respectively the total number of paths that have been saved during the execution of Algorithm~\ref{algo:shortestpath} for a possible extension, and the total number of paths that have been discarded by lower bounds. The number of paths discarded is in general greater than the number of saved paths, which shows that the lower bounds are important for the efficiency of Algorithm~\ref{algo:shortestpath}.

\subsection{Relevance of the ``reduction'' trick}\label{subsec:improve-reduc}

Experiments on the same instances, but without the ``reduction'' trick, have been carried out. %The settings are the same except that the number of columns added at each iteration is $21$ (if any), and that there is of course no $K$. The number $21$ has been determined experimentally to lead to the best possible behavior; yet, despite this choice,
Table~\ref{tab:cpualgo1} gathers the results of these experiments. The information reported in this table is the same as that of Table~\ref{tab:cpualgo2}. The ``reduction'' trick reduces the total computational times for all sets of instances, except $I_1$ and $I_4$. For $I_1$, this is not significant because these are small instances. For $I_4$, a closer inspection on the results has shown that the ``reduction'' trick reduces the total computational time for four out of the five instances: for only one instance, the total computational time is larger (by a huge amount) when the ``reduction'' trick is used; this is actually due to a very high number of columns that have been generated, which makes the last call to the solver very costly (in Table~\ref{tab:cpualgo2}, two instances of $I_4$ are not solved to optimality; only one in Table~\ref{tab:cpualgo1}).

\begin{table}
\caption{Results of the overall method without the ``reduction'' trick.}
\centering
\resizebox{\textwidth}{!}{\begin{tabular}{lrrrrrrrrrrrrrrrr}
\toprule
& \multicolumn{8}{c}{Average} &
\multicolumn{8}{c}{Worst} \\
\cmidrule(r{4pt}){2-9} \cmidrule(l){10-17}
& \multicolumn{1}{c}{Total} & \multicolumn{1}{c}{Master}  & \multicolumn{1}{c}{Pricing}  & \multicolumn{1}{c}{First} & \multicolumn{1}{c}{Pricing} & \multicolumn{1}{c}{Second} & \multicolumn{1}{c}{Most sen.} & \multicolumn{1}{c}{Gap} & \multicolumn{1}{c}{Total} & \multicolumn{1}{c}{Master}  & \multicolumn{1}{c}{Pricing}  & \multicolumn{1}{c}{First} & \multicolumn{1}{c}{Pricing} & \multicolumn{1}{c}{Second} & \multicolumn{1}{c}{Most sen.} & \multicolumn{1}{c}{Gap}\\
& \multicolumn{1}{c}{(s)}  &  \multicolumn{1}{c}{pb.} & \multicolumn{1}{c}{pbs.} & \multicolumn{1}{c}{ILLP} & \multicolumn{1}{c}{w/ gap} & \multicolumn{1}{c}{ILLP}& \multicolumn{1}{c}{p. w/ gap} & \multicolumn{1}{c}{($\%$)}  & \multicolumn{1}{c}{(s)} & \multicolumn{1}{c}{pb.} & \multicolumn{1}{c}{pbs.}& \multicolumn{1}{c}{ILLP} & \multicolumn{1}{c}{w/ gap} & \multicolumn{1}{c}{ILLP} & \multicolumn{1}{c}{p. w/ gap} & \multicolumn{1}{c}{($\%$)}\\
&  &  \multicolumn{1}{c}{(s)}   & \multicolumn{1}{c}{(s)}  &\multicolumn{1}{c}{(s)}  & \multicolumn{1}{c}{(s)} & \multicolumn{1}{c}{(s)}  & \multicolumn{1}{c}{(\# non-opt.)} &   & & \multicolumn{1}{c}{(s)} & \multicolumn{1}{c}{(s)} & \multicolumn{1}{c}{(s)} & \multicolumn{1}{c}{(s)} & \multicolumn{1}{c}{(s)} &  & \\
\midrule
$I_1$ &  7.1 & 0.8 & 1.9 & 1.6 & 0.1 & 2.7  & -& - &16.1& 1.0 &   2.0 &      2.5 &  0.2 & 10.4&-&-\\

$I_2$ & 113.2 & 3.4 & 94.6 & 15.2 & 0 & 0  & -& -& 355.4 &4.9 & 320.6 &      29.8 & 0&0&-&-\\

$I_3$ &  266.2 & 19.7 & 154.7 & 91.8 & 0 & 0  & - & -&384.0 &      20.1 &      249.4 &      114.4 &            0 &            0&-&-\\

$I_4$ &   1,149.7 & 62.1 & 279.7 & 279.5 & 1.7 & 526.7  &  38.3(3) & 62.8 & 1666.2 &      76.8 &      322.3 &      284.5 &        4.7 &      977.8 & 20 & 13.3\\

$I_5$ & 2,240.2 & 103.4 & 1,267.2 & 546.4 & 0.4 & 322.8 & 53.0(1) & 8.3& 4,310.8 &      113.5 &      2,048.1 &       533.2 &        2.0 &      1,614.1 & 53 & 8.3 \\

$I_6$ &   4,431.6 & 155.6 & 2,339.9 & 943.3 & 4.0 & 988.7 & 41.0(2) & 20.2&10,088.9 & 112.1 &6,168.0 &      1,467.9 &        4.6 &      2,336.2  & 56 & 33.3  \\

$I_7$ & 3,336.6 & 246.5 & 996.5 & 919.8 & 2.1 & 1,171.7  & 50.3(3) & 25.6& 4,508.7 &       268.9 &      1,030.4 &      884.0 &        4.2 &      2,321.2 & 71 & 25.0  \\
$I_8$&153,748.0 & 3,128.5 & 56,591.8 & 50,009.2 & 113.4 & 43,905.6 & - & - & - &-&-&-&-&-&-&-\\
\bottomrule
\end{tabular}}
\label{tab:cpualgo1}
\end{table}

%The method is not able to solve instances $I_5$ and $I_7$. In these latter instances, many columns have been generated at step~\ref{pricing2} which resulted in an ``out of memory'' error (the computational time of the last column of Table~\ref{tab:cpualgo1} corresponds to the first ILLP). Table~\ref{tab:cpualgo1} shows that the computational time required to solve the remaining instances is always much larger than those reported in Table~\ref{tab:cpualgo2}: it is from twice to ten times slower.

%Using the ``reduction'' trick improves as expected the computational time spent solving the pricing problem. Maybe more surprisingly, it also decreases the time spent solving the master problem: it might be related to the fact that the ``reduction'' trick leads to the addition of the same schedules for the pilots, while without it, there are more variety; but the concrete reason remains elusive.

\section{Conclusion}

In this work, we propose a lexicographic column generation approach for the PBS that deals directly with the lexicographic objective instead of using a sequence of column generation.
Our algorithm can solve to optimality instances with up to $100$ pilots in around $1$ hour, which makes it competitive with state-of-the-art approaches.
Beyond the generalization of the different algorithmic components of a column generation to the lexicographic setting, two elements are crucial for the performance.

The first element is the ability to find efficiently a lexicographic basic optimal solution of the master problem.
%Numerical experiments indeed show that it is the most time-consuming phase of the algorithm.
Given the importance of the implementation for simplex-like algorithms, this requires using state-of-the-art solvers.
Yet, even if such solvers provide the possibility of computing a lexicographic optimal solution using a simple sequential approach, they do not return a lexicographic primal-dual basis, which is a crucial ingredient in any column generation. We have established in Section~\ref{sec:master} that such primal-dual bases do exist and can be computed via a more elaborate sequential approach, which shows the desired performance.

The second element is the ability to solve efficiently the lexicographic pricing problems. A specificity of the PBS formulation considered in this work is that each objective corresponds to one block of the Dantzig--Wolfe decomposition. %  (one block).
With the ``reduction'' trick introduced in Section~\ref{sec:simultaneousPricing}, the resolution of a single lexicographic resource-constrained longest path problem gives the optimal solution of most pilots pricing problems simultaneously.
We have seen in numerical experiments that this trick indeed improves the performance of the overall method.

\subsection*{Acknowledgments}
The authors are grateful to Boris Detienne for bringing to their attention the papers by \cite{crowder_solving_1983} and~\cite{dantzig_solution_1954} mentioned in Section~\ref{subsec:prel}. They are also thankful to Mohand Ait Alamara and Séverine Bonnechère at Air France for their explanations concerning the constraints defining feasible schedules and for providing concrete instances. %valuable assistance through out the project.

\bibliographystyle{unsrtnat}
\bibliography{biblio}

\end{document}